\documentclass[12pt, A4paper]{article}
\def\aa{\alpha}   \def\ba{\beta}       
\def\ga{\gamma}      \def\ta{\theta}   
   
   \def\Sa{\Sigma}     
  
\def\iy{\infty}   \def\fc{\frac} \def\st{\sqrt}
  \def\pl{\partial}
  
\def\na{\nabla}

\def\ev{\equiv} 

\def\mb{\mathbb}        
\def\mm{\mathrm}
\def\lt{\left}       \def\rt{\right}
\def\lae{\langle}    \def\rae{\rangle}
\def\ls{\limits}     \def\mx{\mbox}
             
\def\ue{\underline}       

\def\de{\displaystyle}
\usepackage{amsmath, amssymb, amsthm, psfrag, graphicx}
\theoremstyle{plain}
\newtheorem{thm}{\textrm{Theorem}}
\newtheorem*{thm1}{\textrm{Theorem}}

\newtheorem{prop}[thm]{\textrm{Proposition}}

\theoremstyle{definition}

\theoremstyle{remark}

\newtheoremstyle{note}
     {3pt}
     {3pt}
     {}
     {}
     {\itshape}
     {}
     {.5em}
     {}
\theoremstyle{note}

\author{Kuo-Wei Lee}
\title{Existence of constant mean curvature foliation in the extended Schwarzschild spacetime}
\begin{document}
\fontsize{12}{18pt plus.5pt minus.4pt}\selectfont
\maketitle
\begin{abstract}
We construct a $T$-axisymmetric, spacelike, spherically symmetric, constant mean curvature hypersurfaces foliation in the Kruskal extension with properties that
the mean curvature varies in each slice and ranges from minus infinity to plus infinity.
This family of hypersurfaces extends the {\small CMC} foliation discussions posted by Malec and \'{O} Murchadha in 2009 \cite{MO2}.

\end{abstract}
\section{Introduction}
Spacelike constant mean curvature ({\small CMC}) hypersurfaces in spacetimes are very important objects in general relativity.
They are broadly used in the analysis on Einstein constraint equations
\cite{CYY, L} and in the gauge condition in the Cauchy problem of the Einstein equations \cite{AM, CYR}.
In addition, {\small CMC} foliation property is identified as the absolute time function in cosmological spacetimes \cite{Y}.

In the field of relativistic cosmology, how to define a canonical absolute time function is an important issue.
At first, people considered the cosmological time function,
which is defined by the supremum of the lengths of all past-directed timelike curves starting at some point.
The idea of the cosmological time function is natural, but bad regularity is its drawback.

In 1971, York \cite{Y} suggested the {\small CMC} time function, which is a real value function $f(x)$
defined on a spacetime such that every level set $\{f(x)=H\}$ is a Cauchy hypersurface with constant mean curvature $H$.
In cosmological spacetimes, by the maximum principle, if the {\small CMC} time function exists, then it is unique.
Furthermore, {\small CMC} time function has better regularity than the cosmological time function.
These properties indicate another viewpoint of the absolute time function.
By definition, if the {\small CMC} time function exists,
the spacetime is foliated by Cauchy hypersurfaces with constant mean curvature,
and the mean curvature of these Cauchy hypersurfaces increases with time.
This phenomena leads us to concern about the {\small CMC} foliation problem in spacetimes.

Many {\small CMC} foliation results are proved for cosmological spacetimes (spatically compact spacetimes) with constant sectional curvature in
\cite{ABBZ} and in its references.
However, {\small CMC} foliation property are not well-understood for spatically noncompact spacetimes such as the Schwarzschild spacetime (Kruskal extension),
which is the simplest model of a universe containing a star.
In \cite{MO}, Malec and \'{O} Murchadha constructed a family of $T$-axisymmetric, spacelike, spherically symmetric,
constant mean curvature ({\small TSS-CMC}) hypersurfaces in the Kruskal extension,
where each slice has the same mean curvature, and they conjectured this family foliates the Kruskal extension.
In \cite{KWL}, the author used the shooting method and Lorentzian geometric analysis to prove the existence and uniqueness of the Dirichlet problem for
{\small SS-CMC} equation with symmetric boundary data in the Kruskal extension.
As an application, the author completely proved the Malec and \'{O} Murchadha's {\small TSS-CMC} foliation conjecture.

In \cite{MO2}, Malec and \'{O} Murchadha discussed different {\small TSS-CMC} foliation property.
They asked whether there is a {\small TSS-CMC} foliation with varied constant mean curvature in each slice.
One result is that if the relation between the mean curvature $H$ and the {\small TSS-CMC} hypersurface parameter $c$ are proportional, that is, $c=-8M^3H$,
then there is a family of {\small TSS-CMC} hypersurfaces so that
$H$ ranges from minus infinity to plus infinity, but all hypersurfaces intersect at the bifurcation sphere (the origin in the Kruskal extension).

In this paper, we will construct another family of {\small TSS-CMC} hypersurfaces with varied constant mean curvature in each slice.
If $H$ and $c$ have a nonlinear relation,
then there exists a {\small TSS-CMC} hypersurfaces foliation in the Kruskal extension.
The statement of the main theorem is the following:
\begin{thm1}
There exists a family of hypersurfaces $\{\Sa_{H(c),c}\}$, $c\in\mb{R}$ in the Kruskal extension satisfying the following properties:
\begin{itemize}
\item[\rm(a)] Every $\Sa_{H(c),c}$ is a $T$-axisymmetric, spacelike, spherically symmetric, constant mean curvature hypersurface.\\[-10mm]
\item[\rm(b)] Any two hypersurfaces in $\{\Sa_{H(c),c}\}$ are disjoint.\\[-10mm]
\item[\rm(c)] Every point $(T',X')$ in the Kruskal extension belongs to $\Sa_{H(c'),c'}$ for some $c'\in\mb{R}$.\\[-10mm]
\item[\rm(d)] When $\{\Sa_{H(c),c}\}$ foliates the Kruskal extension from the bottom to the top,
the corresponding constant mean curvature $H$ ranges from $-\iy$ to $\iy$ and the parameter $c$ ranges from $\iy$ to $-\iy$.\\[-10mm]
\item[\rm(e)] $\{\Sa_{H(c),c}\}$ is invariant under the reflection with respect to the $X$-axis.
\end{itemize}
\end{thm1}

It is remarkable that by similar argument,
we can construct many different {\small TSS-CMC} hypersurfaces foliations with varied $H$ and with $X$-axis symmetry so that the
{\small TSS-CMC} foliation in the Kruskal extension is not unique.
Furthermore, we can also get {\small TSS-CMC} hypersurfaces foliations with varied $H$ but without $X$-axis symmetry.
By Lorentzian isometry, there are {\small SS-CMC} hypersurfaces foliations with varied $H$ but without $T$-axis symmetry.

The organization of this paper is as follows.
In section~\ref{Preliminary}, we first give a brief introduction to the Schwarzschild spacetime and Kruskal extension,
and then we summarize results of the {\small TSS-CMC} hypersurfaces in the Kruskal extension in order to construct a {\small TSS-CMC} foliation.
The main theorem is stated and proved in section~\ref{Proof}.
Some discussions about {\small TSS-CMC} foliation property are in section~\ref{discussion}.

The author would like to thank Yng-Ing Lee, Mao-Pei Tsui, and Mu-Tao Wang for their interests and discussions.
The author is supported by the MOST research grant 103-2115-M-002-013-MY3.

\section{Preliminary} \label{Preliminary}
\subsection{The Kruskal extension}
In this paper, we mainly focus on the Kruskal extension, which is the maximal analytic extended Schwarzschild spacetime.
The Schwarzschild spacetime is a $4$-dimensional time-oriented Lorentzian manifold equipped with the metric
\begin{align*}
\mathrm{d}s^2=-\left(1-\frac{2M}r\right)\mathrm{d}t^2+\frac1{\left(1-\frac{2M}r\right)}\,\mathrm{d}r^2
+r^2\,\mathrm{d}\theta^2+r^2\sin^2\theta\,\mathrm{d}\phi^2,
\end{align*}
where $M>0$ is a constant.
The metric is not defined at $r=2M$, but in fact it is a coordinate singularity.
That is, after coordinates change, the metric is smooth at $r=2M$:
\begin{align}
\mathrm{d}s^2&=\frac{16M^2\mathrm{e}^{-\frac{r}{2M}}}{r}(-\mathrm{d}T^2+\mathrm{d}X^2)+r^2\,\mathrm{d}\theta^2+r^2\sin^2\theta\,\mathrm{d}\phi^2,
\label{KruskalMetric}
\end{align}
where
\begin{align}
\left\{
\begin{array}{l}
\displaystyle(r-2M)\,\mathrm{e}^{\frac{r}{2M}}=X^2-T^2\\
\displaystyle\frac{t}{2M}=\ln\left|\frac{X+T}{X-T}\right|. \label{trans}
\end{array}
\right.
\end{align}
The Kurskal extension is the union of two Schwarzschild spacetimes equipped with the extended metric (\ref{KruskalMetric}).
Figure~\ref{KruskalSimple} points out the correspondences between the Kruskal extension (left figure, $T$-$X$ plane)
and Schwarzschild spacetimes (right figure, $t$-$r$ plane).
We refer to Wald's book \cite{W} or the paper \cite{LL1} for more discussions on the Kruskal extension.

\begin{figure}[h]
\psfrag{A}{\tt I}
\psfrag{B}{\tt I$\!$I}
\psfrag{C}{\tt I$\!$I'}
\psfrag{D}{\tt I'}
\psfrag{E}{$U$}
\psfrag{F}{$V$}
\psfrag{X}{$X$}
\psfrag{T}{$T$}
\psfrag{r}{$r$}
\psfrag{t}{$t$}
\psfrag{M}{\tiny$2M$}
\psfrag{P}{\tiny $\partial_T$}
\centering
\includegraphics[height=50mm,width=87mm]{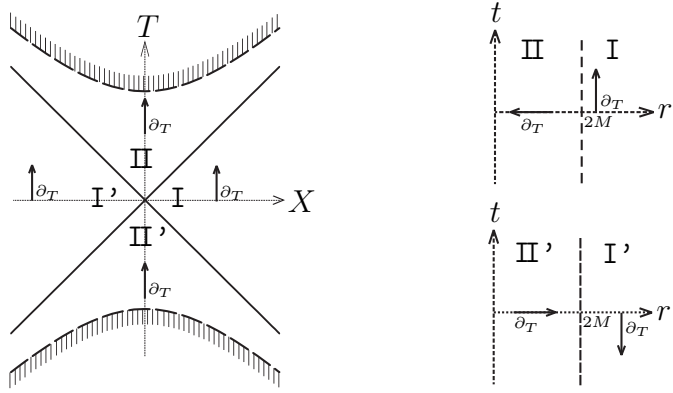}
\caption{The Kruskal extension and the Schwarzschild spacetimes.} \label{KruskalSimple}
\end{figure}

Remark that each point in the Kruskal $T$-$X$ plane or the Schwarzschild $t$-$r$ plane is topologically a sphere $\mb{S}^2$,
which is parameterized by $\ta$ and $\phi$.
In this article, we are interested in the spherically symmetric hypersurfaces.
It implies that every such hypersurface is a curve in both $T$-$X$ and $t$-$r$ plane.
For notation convenience, we will ignore parameters $\ta$ and $\phi$ in this paper.

We take $\partial_T$ as a future-directed timelike vector field in the Kruskal extension,
which is also pointed out in Figure~\ref{KruskalSimple}.
Once $\pl_T$ is chosen, for a spacelike hypersurface $\Sa$,
we will choose $\vec{n}$ as the future-directed unit normal vector of $\Sa$ in the Kruskal extension,
and the mean curvature $H$ of $\Sa$ is defined by $H=\fc13g^{ij}\lae\na_{e_i}\vec{n},e_j\rae$,
where $\{e_i\}_{i=1}^3$ is a basis on $\Sa$.

\subsection{$T$-axisymmetric, spacelike, spherically symmetric, constant mean curvature hypersurfaces in the Kruskal extension}
Let $\Sa:(T=F(X),X)$ be a spacelike, spherically symmetric, constant mean curvature ({\small SS-CMC} for short) hypersurface in the Kruskal extension.
In \cite{LL1}, we computed the {\small SS-CMC} equation:
\begin{align}
F''(X)&+\mbox{e}^{-\frac{r}{2M}}\left(\frac{6M}{r^2}-\frac1r\right)(-F(X)+F'(X)X)(1-(F'(X))^2)\notag\\
&+\frac{12HM\mbox{e}^{-\frac{r}{4M}}}{\sqrt{r}}(1-(F'(X))^2)^{\frac32}=0, \label{CMCequation}
\end{align}
where the spacelike condition is $1-(F'(X))^2>0$, and $r=r(T,X)=r(F(X),X)$
satisfies the equation (\ref{trans}), namely, $(r-2M)\,\mathrm{e}^{\frac{r}{2M}}=X^2-T^2=X^2-(F(X))^2$.

Since the equation (\ref{CMCequation}) contains $r$, which is a nonlinear relation between $T=F(X)$ and $X$,
it is challenging to get results from the equation (\ref{CMCequation}) such as the existence, uniqueness, and behavior of the solution.
Instead of dealing with the equation~(\ref{CMCequation}), in papers \cite{KWL, LL1} and \cite{LL2},
we solved and analyzed the {\small SS-CMC} equation in each Schwarzschild spacetime region.
Suppose that $\Sa:(t=f(r),r)$ is an {\small SS-CMC} hypersurface in the Schwarzschild spacetime,
then $f(r)$ satisfies the following {\small SS-CMC} equation:
\begin{align}
f''+\left(\left(\frac1{h}-(f')^2h\right)\left(\frac{2h}{r}+\frac{h'}2\right)+\frac{h'}{h}\right)f'\pm 3H\left(\frac1{h}-(f')^2h\right)^{\frac32}=0,
\label{SSCMCeqn}
\end{align}
where $h(r)=1-\fc{2M}{r}$, and the spacelike condition is $\frac1{h}-(f')^2h>0$.
Remark that the choice of $\pm$ signs in (\ref{SSCMCeqn})
depends on different regions and different pieces of {\small SS-CMC} hypersurfaces.
Since the equation (\ref{SSCMCeqn}) is a second order ordinary differential equation,
the solution is solved explicitly, and we can completely characterize {\small SS-CMC} hypersurfaces in the Kruskal extension through relations (\ref{trans}).

Here we summarize results in \cite{LL1} and \cite{LL2} about the construction of the $T$-axisymmetric {\small SS-CMC} ({\small TSS-CMC}) hypersurfaces.
These results will be used for further discussions in this article.
In paper \cite{LL1}, the solution of the equation (\ref{SSCMCeqn}) in the Schwarzschild interior which maps to the Kruskal extension {\tt I\!I'} is
\begin{align*}
f(r;H,c,\bar{c})=\lt\{
\begin{array}{ll}
\de\int_{r_0}^r\fc{l(x;H,c)}{-h(x)\st{l^2(x;H,c)-1}}\,\mm{d}x+\bar{c}, & \mx{if } f'(r)>0 \\[5mm]
\de\int_{r_0}^r\fc{l(x;H,c)}{h(x)\st{l^2(x;H,c)-1}}\,\mm{d}x+\bar{c}, & \mx{if } f'(r)<0,
\end{array}\rt.
\end{align*}
where $r_0$ is a point in the domain of $f(r)$, $l(r;H,c)=\fc1{\st{-h(r)}}\lt(Hr+\fc{c}{r^2}\rt)$,
and $c, \bar{c}$ are two constants of integration.
Here we require $l(r;H,c)>1$ so that the function $f(r)$ is meaningful, and it is equivalent to $c>-Hr^3+r^{\fc32}(2M-r)^{\fc12}$,
so it is natural to define the function
\begin{align*}
\tilde{k}_H(r)=-Hr^3+r^{\fc32}(2M-r)^{\fc12}
\end{align*}
to analyze the domain of the solution $f(r)$.

Now we look at the case $H\leq 0$.
Given $H$, in Figure~\ref{CMCpaperFunctionK}, the function $\tilde{k}_H(r)$ has a maximum value $C_H$ at $r=R_H$.
Denote the increasing part and decreasing part of the function $\tilde{k}_H(r)$ by $\tilde{k}^+_H(r)$ and $\tilde{k}^-_H(r)$, respectively.
For $c\in(0,C_H)$, the solution of $\tilde{k}^+_H(r)=c$ is denoted by
$r=\tilde{r}_{H,c}^+$, then $(0,\tilde{r}_{H,c}^+]$ is the domain of the {\small SS-CMC} solution $f(r)$.
Remark that $r=\tilde{r}_{H,c}^+$ belongs to the domain of $f(r)$ because the behavior $f'(r)\sim O((r-\tilde{r}_{H,c}^+)^{-\fc12})$ implies that
$f(\tilde{r}_{H,c}^+)$ is a finite value.
Consider the {\small SS-CMC} hypersurface which is the union of two graphs of $t=f(r)$, where one satisfies $f'(r)>0$ and the other satisfies $f'(r)<0$,
and two graphs are smoothly joined at the point $(t,r)=(0,r_{H,c}^+)$.
This {\small SS-CMC} hypersurface is symmetric about $t=0$.
Since $t=0$ in the Schwarzschild interior is the $T$-axis in the Kruskal extension {\tt I\!I'},
this hypersurface maps to a {\small TSS-CMC} hypersurface $\tilde{\Sa}_{H,c}^+$ in the Kruskal extension {\tt I\!I'},
and $\tilde{\Sa}_{H,c}^+$ intersects the $T$-axis at $T=-\st{2M-\tilde{r}_{H,c}^+}\,\mm{e}^{\fc{\tilde{r}_{H,c}^+}{4M}}$.
See Figure~\ref{CMCpaperFunctionK}.

\begin{figure}[h]
\centering
\psfrag{X}{$X$}
\psfrag{T}{$T$}
\psfrag{r}{$r$}
\psfrag{t}{$t$}
\psfrag{A}{\small$R_H$}
\psfrag{B}{$C_H$}
\psfrag{C}{$c$}
\psfrag{D}{\small$\tilde{r}_{H,c}^+$}
\psfrag{J}{\small$\tilde{r}_{H,c}^-$}
\psfrag{E}{$\tilde{k}_H^+(r)$}
\psfrag{F}{$\tilde{k}_H^-(r)$}
\psfrag{G}{\small$2M$}
\psfrag{H}{$\tilde{\Sa}_{H,c}^+$}
\psfrag{I}{$\tilde{\Sa}_{H,c}^-$}
\psfrag{K}{$\tilde{\Sa}_{H,C_H}$}
\includegraphics[height=48mm,width=115mm]{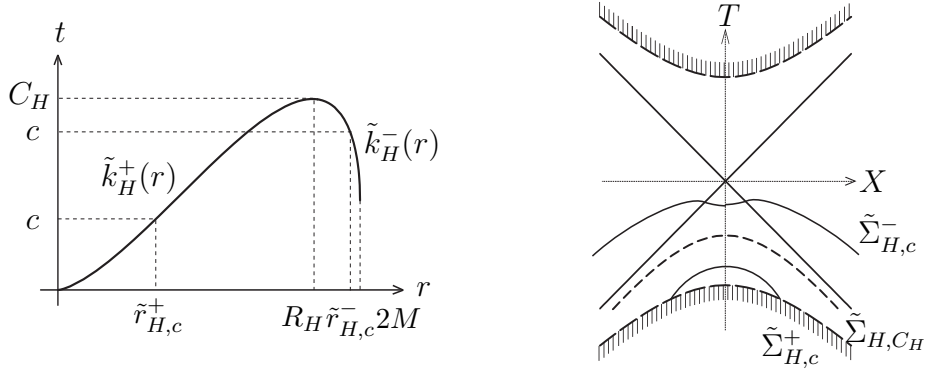}
\caption{Each point on the graph of $\tilde{k}_H(r)$ determines a {\small TSS-CMC}
hypersurface and its $T$-intercept in the Kruskal extension.} \label{CMCpaperFunctionK}
\end{figure}

For $c\in(-8M^3H,C_H)$, the solution of $\tilde{k}^-_H(r)=c$ is denoted by $r=\tilde{r}_{H,c}^-$,
then $[\tilde{r}_{H,c}^-,\iy)$ is the domain of the {\small SS-CMC} solution $f(r)$.
Remark that $f(r)$ is defined at $r=2M$ in the sense of Kruskal extension,
and $\tilde{r}_{H,c}^-$ belongs to the domain because of $f'(r)\sim O((r-r_{H,c}^-)^{-\fc12})$.
The union of graphs of functions $t=f(r)$ which are smoothly joined at the point $(t,r)=(0,r_{H,c}^-)$
maps to a {\small TSS-CMC} hypersurface $\tilde{\Sa}_{H,c}^-$ in the Kruskal extension {\tt I}, {\tt I\!I}, and {\tt I'}.
Furthermore, $\tilde{\Sa}_{H,c}^-$ intersects the $T$-axis at $T=-\st{2M-\tilde{r}_{H,c}^-}\,\mm{e}^{\fc{\tilde{r}_{H,c}^-}{4M}}$.

The dotted curve between $\tilde{\Sa}_{H,c}^+$ and $\tilde{\Sa}_{H,c}^-$ in Figure~\ref{CMCpaperFunctionK}
is the {\small TSS-CMC} hypersurface $\tilde{\Sa}_{H,C_H}$, which corresponds to the point with maximum value of $\tilde{k}_H(r)$.
The hypersurface $\tilde{\Sa}_{H,C_H}$ is a hyperbola $X^2-T^2=(R_H-2M)\mm{e}^{\fc{R_H}{2M}}$ with $T<0$
in the Kruskal extension {\tt I\!I'}, and it is a cylindrical hypersurface $r=R_H$ in the Schwarzschild interior.

From the above discussion,
we establish an one-to-one correspondence from each point on the graph of $\tilde{k}_H(r)$ to a {\small TSS-CMC} hypersurface $\tilde{\Sa}_{H,c}$.

\subsection{The construction of TSS-CMC foliation with varied $H$} \label{subsectionconstruction}
In order to construct a {\small TSS-CMC} hypersurfaces foliation with varied $H$ in each slice,
we will view $H$ as a variable and thus consider the two variables function
\begin{align*}
\tilde{k}(H,r)=-Hr^3+r^{\fc32}(2M-r)^{\fc12},
\end{align*}
where $r\in[0,2M]$ and $H\leq 0$.
Here we only consider $H\leq 0$ because we will use the symmetry property to get $H\geq 0$ part.
First we prove the following Proposition.

\begin{prop} \label{Proposition1}
For the function $\tilde{k}(H,r)=-Hr^3+r^{\fc32}(2M-r)^{\fc12}$ where $r\in[0,2M]$ and $H\leq 0$,
there exists a function $y(r)$ defined on $(0,2M]$ such that
\begin{align}
\lt\{
\begin{array}{l}
\de\fc{\mm{d}y}{\mm{d}r}\neq\fc{3y}{r}+\fc{r^{\fc12}(-3M+r)}{(2M-r)^{\fc12}} \\[3mm]
\de\fc{\mm{d}y}{\mm{d}r}<0 \mx{ for all } r\in(0,2M) \\[2mm]
\de y(2M)=0 \mx{ and } \lim\ls_{r\to 0^+}y(r)=\iy.
\end{array}
\rt. \label{Conditions}
\end{align}
\end{prop}

\begin{proof}
First of all, we compute
\begin{align*}
\fc{\pl \tilde{k}}{\pl r}(H,r)
=-3Hr^2+\fc{r^{\fc12}(3M-2r)}{(2M-r)^{\fc12}}
=\fc{3y}{r}+\fc{r^{\fc12}(-3M+r)}{(2M-r)^{\fc12}}.
\end{align*}
Here we replace $H$ with $y$ and $r$ by the relation $y=-Hr^3+r^{\fc32}(2M-r)^{\fc12}$ in the last equality.
To find the function $y(r)$, it suffices to find a function $h(r)>0$ such that
\begin{align}
\lt\{
\begin{array}{l}
\de\fc{\mm{d}y}{\mm{d}r}-\fc{3y}{r}=\fc{r^{\fc12}(-3M+r)}{(2M-r)^{\fc12}}-h(r) \\
\de y(2M)=0.
\end{array}\rt. \label{eqnhfunction}
\end{align}
When multiplying the integrating factor $\mm{e}^{\int-\fc3r\,\mm{d}r}=r^{-3}$ on both sides of the differential equation (\ref{eqnhfunction}), it becomes
\begin{align*}
\fc{\mm{d}}{\mm{d}r}\lt(r^{-3}y(r)\rt)&=\fc{-3M+r}{r^{\fc52}(2M-r)^{\fc12}}-\fc{h(r)}{r^3}.
\end{align*}
After integration, the function $y(r)$ is solved:
\begin{align*}
y(r)=r^{\fc32}(2M-r)^{\fc12}+r^3\int_r^{2M}\fc{h(x)}{x^3}\,\mm{d}x.
\end{align*}
Next, we calculate
\begin{align*}
y'(r)=\fc{r^{\fc12}(3M-2r)}{(2M-r)^{\fc12}}+3r^2\int_r^{2M}\fc{h(x)}{x^3}\,\mm{d}x-h(r).
\end{align*}
Consider the function $h(r)$ is of the form $h(r)=Cr^{-p}$, where $C$ and $p$ are positive numbers to be determined. Then
\begin{align*}
y'(r)=\fc{r^{\fc12}(3M-2r)}{(2M-r)^{\fc12}}-C\lt(\fc{3}{(p+2)(2M)^{p+2}}+\fc{(p-1)}{(p+2)r^p}\rt).
\end{align*}
Since the function $g(r)=\fc{r^{\fc12}(3M-2r)}{(2M-r)^{\fc12}}$ has a global maximum value
$g(r_*)=\st{6\st{3}-9}M$ at $r_*=\fc{(3-\st{3})M}{2}$,
we can choose any value $p>1$ and then choose the constant $C$ large enough such that $y'(r)<0$ for all $r\in(0,2M)$.

Finally, we check the limit behavior:
\begin{align*}
\lim_{r\to 0^+} y(r)
&=\lim_{r\to 0^+}\lt(r^{\fc32}(2M-r)^{\fc12}+r^3\int_r^{2M}\fc{C}{x^{p+3}}\,\mm{d}x\rt) \\
&=C(p+2)\lim_{r\to 0^+}\lt(\fc{1}{r^{p-1}}-\fc{r^3}{(2M)^{p+2}}\rt)
\to\iy.
\end{align*}
\end{proof}

In the following paragraphs, we will use the notation $\tilde{k}_H(r)$ if we consider the function $\tilde{k}(H,r)$ with fixed $H$.
From Proposition~\ref{Proposition1},
we find a strictly decreasing function $y(r)$ such that the equation $y(r)=\tilde{k}_{H}(r)$ has a unique solution for every $H\leq 0$.
Figure~\ref{CMCpaperFamilyKtoHypersurfaces}~(a) illustrates the curve $\ga$, which is the graph of $y(r)$,
and we set the curve $\ga(c)$ with parameter $c$ by $c=y(r)$.
Since $y(r)$ is strictly decreasing, we have $r=y^{-1}(c)$, and the mean curvature can be expressed as $H(c)$ by the relation $c=-Hr^3+r^{\fc32}(2M-r)^{\fc12}$.
Thus there is an one-to-one correspondence from each point on $\ga(c)$ to a {\small TSS-CMC} hypersurface $\Sa_{H(c),c}$,
where $\Sa_{H(c),c}$ intersects the $T$-axis at $T=-\st{2M-y^{-1}(c)}\,\mm{e}^{\fc{y^{-1}(c)}{4M}}$,
as Figure~\ref{CMCpaperFamilyKtoHypersurfaces}~(c) showed.

\begin{figure}[h]
\psfrag{r}{$r$}
\psfrag{t}{$t$}
\psfrag{X}{$X$}
\psfrag{T}{$T$}
\psfrag{G}{$\ga(c)$}
\psfrag{A}{$\aa$}
\psfrag{P}{$\ga(-c)$}
\psfrag{B}{-$\aa$}
\psfrag{a}{(a)}
\psfrag{b}{(b)}
\psfrag{c}{(c)}
\psfrag{D}{$c$}
\psfrag{C}{$C$}
\psfrag{H}{\footnotesize $\stackrel{\de\Sa_{H(c),c}}{c>C}$}
\psfrag{K}{$\Sa_{H(C),C}$}
\psfrag{I}{\footnotesize $\stackrel{\de\Sa_{H(c),c}}{0<c<C}$}
\includegraphics[height=50mm,width=136mm]{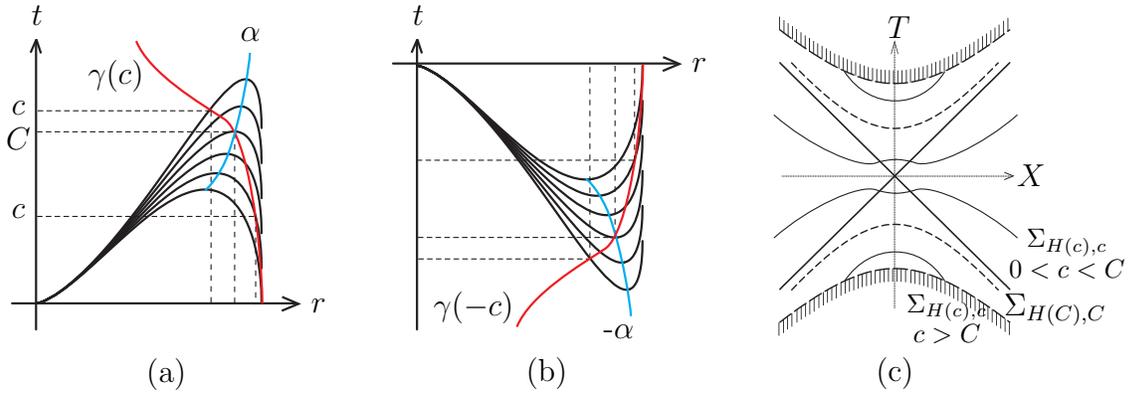}
\caption{The curve $\ga(c)$ is the graph of a decreasing function and each point on $\ga(c)$ corresponds to a {\small TSS-CMC} hypersurface $\Sa_{H(c),c}$
in the Kruskal extension.} \label{CMCpaperFamilyKtoHypersurfaces}
\end{figure}

In Figure~\ref{CMCpaperFamilyKtoHypersurfaces}~(a), we trace another curve $\aa$.
The curve $\aa$ consists of all points $(t,r)$ satisfying $t=\max\ls_{r\in[0,2M]}\tilde{k}_H(r)$ for every $H\leq 0$.
It is easy to know that the curve $\aa$ is a graph of an increasing function so that the curve $\ga$ and $\aa$ intersects once,
and we denote the intersection point by $(C,R)$.
The point $(C,R)$ corresponds to the hyperbola $X^2-T^2=(R-2M)\mm{e}^{\fc{R}{2M}}$ with $T<0$,
and it is the dotted curve in Figure~\ref{CMCpaperFamilyKtoHypersurfaces}~(c).

When $c=0$, we get $r=2M$ and $H=0$.
The {\small TSS-CMC} hypersurfaces $\Sa_{H(0),0}$ is a maximal hypersurface passing through $(T,X)=(0,0)$
so that $\Sa_{H(0),0}$ is $T\ev 0$, or $X$-axis.
So far, we have constructed the {\small TSS-CMC} foliation in the region $T\leq 0$.

Next, we consider another two variables function
\begin{align*}
k(H,r)=-Hr^3-r^{\fc32}(2M-r)^{\fc12},
\end{align*}
where $r\in[0,2M]$ and $H\geq 0$.
The function $k(H,r)$ comes from the inequality $l(r;H,c)=\fc{1}{\st{-h(r)}}\lt(-Hr-\fc{c}{r^2}\rt)>1$,
and it will determine the domain of an {\small SS-CMC}
solution in the Schwarzschild interior which maps to the Kruskal extension {\tt I\!I}.
In fact, the function $k(H,r)$ for $H\geq 0$ and $\tilde{k}(H,r)$ for $H\leq 0$ are symmetric about $r$-axis,
so for the construction of {\small TSS-CMC} hypersurfaces in the Kruskal extension $T\geq 0$ part, in Figure~\ref{CMCpaperFamilyKtoHypersurfaces}~(b),
we choose the curve $\ga(-c)$ by the reflection of the curve $\ga(c)$ with respect to the $r$-axis.
Each point on $\ga(-c)$ will one-to-one correspond to a {\small TSS-CMC} hypersurface $\Sa_{H(-c),-c}$ in the Kruskal extension,
and $\Sa_{H(-c),-c}$ intersects $T$-axis at $T=\st{2M-(-y)^{-1}(-c)}\mm{e}^{\fc{(-y)^{-1}(-c)}{4M}}$, as Figure~\ref{CMCpaperFamilyKtoHypersurfaces}~(c) showed.
Furthermore, hypersurfaces $\Sa_{H(-c),-c}$ and $\Sa_{H(c),c}$ are symmetric about $X$-axis.

Finally, we collect {\small TSS-CMC} hypersurfaces $\{\Sa_{H(c),c}\}, c\in\mb{R}$.
Remark that when the parameter $c$ ranges from $\iy$ to $-\iy$, the mean curvature $H$ ranges from $-\iy$ to $\iy$.
In next section, we will show that $\{\Sa_{H(c),c}\}, c\in\mb{R}$, forms a {\small TSS-CMC} foliation in the Kruskal extension.

\section{Existence of TSS-CMC foliation} \label{Proof}
Now we are ready to prove the family $\{\Sa_{H(c),c}\}, c\in\mb{R}$,
we constructed in section~\ref{subsectionconstruction} foliates the Kruskal extension.

\begin{thm}
There exists a family of hypersurfaces $\{\Sa_{H(c),c}\}$, $c\in\mb{R}$, in the Kruskal extension satisfying the following properties:
\begin{itemize}
\item[\rm(a)] Every $\Sa_{H(c),c}$ is a $T$-axisymmetric, spacelike, spherically symmetric, constant mean curvature hypersurface.\\[-10mm]
\item[\rm(b)] Any two hypersurfaces in $\{\Sa_{H(c),c}\}$ are disjoint.\\[-10mm]
\item[\rm(c)] Every point $(T',X')$ in the Kruskal extension belongs to $\Sa_{H(c'),c'}$ for some $c'\in\mb{R}$.\\[-10mm]
\item[\rm(d)] When $\{\Sa_{H(c),c}\}$ foliates the Kruskal extension from the bottom to the top,
the corresponding constant mean curvature $H$ ranges from $-\iy$ to $\iy$ and the parameter $c$ ranges from $\iy$ to $-\iy$.\\[-10mm]
\item[\rm(e)] $\{\Sa_{H(c),c}\}$ is invariant under the reflection with respect to the $X$-axis.
\end{itemize}
\end{thm}

From our construction, properties (a), (d), and (e) are automatically true.
Furthermore, for (b) and (c), it suffices to prove the $T\leq 0$ part because of the symmetry property (e).
\begin{proof}[Proof of property {\rm(b)}]
Here we prove the case if {\small TSS-CMC} hypersurfaces lie between $T=0$ and the hyperbola $X^2-T^2=(R-2M)\mm{e}^{\fc{R}{2M}}$ with $T<0$,
and it is similarly proved if {\small TSS-CMC} hypersurfaces lie below the hyperbola.
Given any two {\small TSS-CMC} hypersurfaces $\Sa_{H(c_1),c_1}$ and $\Sa_{H(c_2),c_2}$ with $c_1<c_2$,
we compare these two hypersurfaces with $\Sa_{H,c}$, which is a {\small TSS-CMC} hypersurface with mean curvature $H=H(c_1)$
and has the same $T$-intercept as $\Sa_{H(c_2),c_2}$.
See Figure~\ref{CMCFoliationProof}~(a).
Since $H=H(c_1)$, two points in the $t$-$r$ plane corresponding to $\Sa_{H,c}$ and
$\Sa_{H(c_1),c_1}$ lie on the same function $\tilde{k}_{H(c_1)}(r)$.
Because $\Sa_{H,c}$ and $\Sa_{H(c_2),c_2}$ have the same $T$-intercept,
two points in the $t$-$r$ plane corresponding to $\Sa_{H,c}$ and $\Sa_{H(c_2),c_2}$ have the same $r$ value.

\begin{figure}[h]
\centering
\psfrag{r}{$r$}
\psfrag{t}{$t$}
\psfrag{X}{$X$}
\psfrag{T}{$T$}
\psfrag{G}{$\ga(c)$}
\psfrag{A}{$\aa$}
\psfrag{a}{(a)}
\psfrag{b}{(b)}
\psfrag{E}{$c_1$}
\psfrag{D}{$c$}
\psfrag{C}{$c_2$}
\psfrag{P}{$\Sa_{H(c_2),c_2}$}
\psfrag{Q}{$\Sa_{H,c}$}
\psfrag{R}{$\Sa_{H(c_1),c_1}$}
\psfrag{H}{\footnotesize $\stackrel{\de\Sa_{H(c),c}}{c>C}$}
\psfrag{K}{$\Sa_{H(C),C}$}
\psfrag{I}{\footnotesize $\stackrel{\de\Sa_{H(c),c}}{0<c<C}$}
\includegraphics[height=55mm,width=104mm]{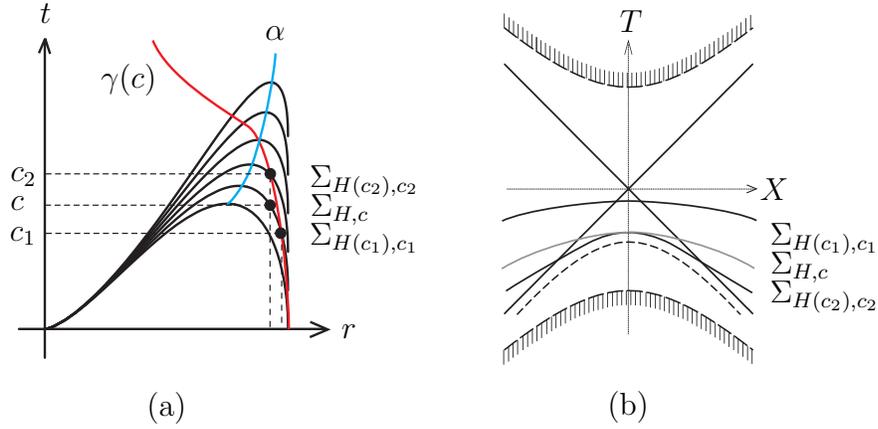}
\caption{$\Sa_{H(c_1),c_1}$ and $\Sa_{H(c_2),c_2}$ are disjoint by comparing with $\Sa_{H,c}$.} \label{CMCFoliationProof}
\end{figure}

Since $c_1<c_2$, we have $H(c_2)<H(c_1)=H$, and it implies $T$ values of $\Sa_{H(c_2),c_2}:(T_{H(c_2),c_2}(X),X)$ and $\Sa_{H,c}:(T_{H,c}(X),X)$
in the Kruskal extension satisfy
\begin{align}
T_{H(c_2),c_2}(X)<T_{H,c}(X)\quad \mx{for all } X\neq 0. \label{comparison}
\end{align}
The inequality (\ref{comparison}) holds because of $\fc{\pl f'}{\pl H}>0$ for all $X\geq 0$ in the Schwarzschild spacetime.
Furthermore, in paper \cite{KWL},
we proved that for every fixed $H\in\mb{R}$,
the curve formed by the union of the graphs of $\tilde{k}_H(r)$ and $k_H(r)$
corresponds to a {\small TSS-CMC} hypersurfaces family $\{\Sa_{H}\}$, and $\{\Sa_{H}\}$ foliates the Kruskal extension.
Since both $\Sa_{H,c}:(T_{H,c}(X),X)$ and $\Sa_{H(c_1),c_1}:(T_{H(c_1),c_1}(X),X)$ belong to $\{\Sa_{H}\}$,
and their $T$-intercepts satisfy $T_{H,c}(0)<T_{H(c_1),c_1}(0)$, we have $T_{H,c}(X)<T_{H(c_1),c_1}(X)$.
Therefore, $T_{H(c_2),c_2}(X)<T_{H(c_1),c_1}(X)$ for all $X$, and hence $\Sa_{H(c_1),c_1}$ and $\Sa_{H(c_2),c_2}$ are disjoint.
See Figure~\ref{CMCFoliationProof}~(b).
\end{proof}

\begin{proof}[Proof of property {\rm(c)}]
Recall that in \cite[Theorem 3 and Theorem 7]{KWL},
we proved the existence and uniqueness of the Dirichlet problem for the {\small TSS-CMC} equation with symmetric boundary data.
In other words, for any fixed $H\in\mb{R}$ and given $(T',X')$ in the Kruskal extension,
there exists a unique {\small TSS-CMC} hypersurface, denoted it by $\Sa_{H,c(H)}$, passing through $(T',X')$ and $(T',-X')$.

Here we prove the case if $(T',X')$ lies between $T=0$ and the hyperbola $X^2-T^2=(R-2M)\mm{e}^{\fc{R}{2M}}$ with $T<0$,
and it is similarly proved if $(T',X')$ lies in other regions.
See Figure~\ref{CMCFoliationProofC}.
When $H=0$, there exists a unique value $c(0)$ such that $(T',X')\in\Sa_{H=0,c(0)}$.
For every $H\leq 0$, we can find a point on the graph of $\tilde{k}_H^-(r)$ corresponding to a {\small TSS-CMC} hypersurface $\Sa_{H,c(H)}$
passing through $(T',X')$.
Set $\ba(c)$ be all such points.
We know that $\ba(c)$ is a continuous curve because Theorem~6 in \cite{KWL} shows these solutions are continuous varied with the mean curvature.
When $H\to-\iy$, we have $\tilde{k}_H^-(r)\to\iy$, so the $t$ value of $\ba(c)$ tends to infinity and $r$ value of $\ba(c)$ tends to $2M$.
By the Intermediate Value Theorem, two curves $\ga(c)$ and $\ba(c)$ must intersect at some point $c'$,
and hance $(T',X')\in\Sa_{H(c'),c'}$.

\begin{figure}[h]
\centering
\psfrag{r}{$r$}
\psfrag{t}{$t$}
\psfrag{G}{$\ga(c)$}
\psfrag{A}{$\aa$}
\psfrag{B}{$\ba(c)$}
\psfrag{E}{$c_1$}
\psfrag{D}{$c(0)$}
\psfrag{C}{$c$}
\psfrag{P}{$\Sa_{H(c'),c'}$}
\psfrag{Q}{$\Sa_{H=0,c(0)}$}
\includegraphics[height=45mm,width=45mm]{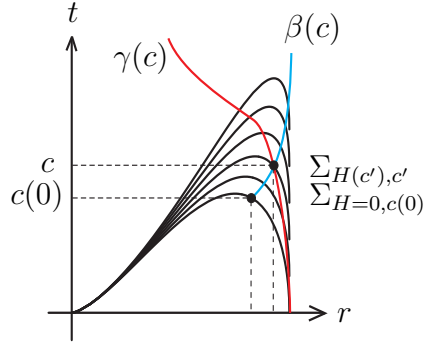}
\caption{Prove the family $\{\Sa_{H(c),c}\}$ covering the Kruskal extension.} \label{CMCFoliationProofC}
\end{figure}
\end{proof}

\section{Discussions} \label{discussion}
We consider a family of $T$-axisymmetric, spacelike, spherically symmetric ({\small TSS-CMC}) hypersurfaces in the Kruskal extension.
The mean curvature is constant on each slice but changes from slice to slice.
First we construct this {\small TSS-CMC} hypersurfaces family with one more symmetry, called $X$-axis symmetry.
That is, after getting {\small TSS-CMC} hypersurfaces in region $T\leq 0$,
we use the reflection with respect to $X$-axis to derive {\small TSS-CMC} hypersurfaces in region $T\geq 0$.
Based on the result of {\small TSS-CMC} hypersurfaces foliation with fixed mean curvature in \cite{KWL},
we prove these {\small TSS-CMC} hypersurfaces foliate the Kruskal extension.

Two functions $\tilde{k}(H,r)$ and $k(H,r)$ play important roles in this foliation argument.
Each point on the graphs of $\tilde{k}(H,r)$ and $k(H,r)$ will one-to-one correspond to a {\small TSS-CMC} hypersurface $\Sa_{H,c}$,
where $H,r,c$ satisfy relations $c=-Hr^3+r^{\fc32}(2M-r)^{\fc12}$ or $c=-Hr^3-r^{\fc32}(2M-r)^{\fc12}$, respectively.
The argument in section~\ref{subsectionconstruction}
indicates that if we find a curve $\ga(c)$ which intersects every $\tilde{k}_H(r)$ and $k_H(r)$ exactly once,
and $\ga(c)$ is the union of two monotonic functions, then we can prove the {\small TSS-CMC} foliation property.

In Proposition~\ref{Proposition1}~(\ref{Conditions}),
there are three conditions the curve $\ga(c)$ should be satisfied.
The first condition $\fc{\mm{d}y}{\mm{d}r}\neq\fc{3y}{r}+\fc{r^{\fc12}(-3M+r)}{(2M-r)^{\fc12}}$
implies that the curve $\ga(c)$ intersects every $\tilde{k}_H(r)$ exactly once.
Second condition $\fc{\mm{d}y}{\mm{d}r}<0 \mx{ for all } r\in(0,2M)$ indicates that
$c$ is decreasing and $H$ is increasing along {\small TSS-CMC} hypersurfaces.
The condition $\lim\ls_{r\to 0^+}y(r)=\iy$ coupled with symmetry $\lim\ls_{r\to 0^+}-y(r)=-\iy$
state that mean curvatures of {\small TSS-CMC} hypersurfaces range from $-\iy$ to $\iy$.
The condition $y(2M)=0$ will impose that the {\small TSS-CMC} hypersurface passing through the bifurcation sphere $(T,X)=(0,0)$ is the maximal hypersurface,
which is $T\ev 0$. In this case, we can use the $X$-axis symmetry to get the whole {\small TSS-CMC} family $\{\Sa_{H,c}\}$.

The first remark is that the existence of $\ga(c)$ is not unique.
This is because $\fc{\mm{d}y}{\mm{d}r}\neq\fc{3y}{r}+\fc{r^{\fc12}(-3M+r)}{(2M-r)^{\fc12}}$ is an open condition.
Therefore, we can find many different {\small TSS-CMC} foliations.

The second remark is that the condition $y(2M)=0$ is more flexible.
In fact, we can consider a more general setting that to solve the function $y(r)$ in Proposition~\ref{Proposition1}
by satisfying (\ref{Conditions}) but replacing the condition $y(2M)=0$ with $y(2M)=A$, where $A\in\mb{R}$.
There still exists a function, denoted by $y_A(r)$ in the general setting.
Let the curve $\ga_A(c)$ be the graph of $y_A(r)$ with parameter $c=y_A(r)$.
So we have derived {\small TSS-CMC} hypersurfaces below $\Sa_{H(A),A}$.
How do we get {\small TSS-CMC} hypersurfaces above $\Sa_{H(A),A}$?
Recall that these {\small TSS-CMC} hypersurfaces are determined by the function $k(H,r)$.
Notice that $k(H,r)=-\tilde{k}(-H,r)$,
so we consider the curve $\ga_{-A}(c)+2A$, which is the curve by moving $\ga_{-A}(c)$ along the $t$ direction by $2A$.
Then $\ga_A(c)$ and $\ga_{-A}(c)+2A$ are joined at $r=2M$ and $\ga_A(c)\cup(\ga_{-A}(c)+2A)$
intersects every $\tilde{k}_H(r)$ and $k_H(r)$ exactly once.
Hence the corresponding {\small TSS-CMC} hypersurfaces family $\{\Sa_{H(c),c}\}, c\in\mb{R}$ foliate the Kruskal extension but it is not $X$-axisymmetric.

Next, every {\small TSS-CMC} foliation can be changed as a {\small SS-CMC} foliation without $T$-axisymmetric property by the Lorentzian isometry.
All of the foliations have the property that mean curvatures range from $-\iy$ to $\iy$.
This phenomena is more close to the definition of the {\small CMC} time function.
However, only {\small TSS-CMC} hypersurfaces across the Kruskal regions {\tt I} and {\tt I'} are Cauchy hypersurfaces.

Finally, our {\small TSS-CMC} foliations construction extends the results and discussions in Malec and \'{O} Murchadha's paper \cite{MO2}.
They considered {\small TSS-CMC} foliations where the mean curvature $H$ and the {\small TSS-CMC} hypersurface parameter $c$ are proportional, that is, $c=-8M^3H$,
then there is a family of {\small TSS-CMC} hypersurfaces so that
$H$ ranges from minus infinity to plus infinity, but all hypersurfaces intersect at the origin in the Kruskal extension.
In order to break the phenomena of intersection, we consider the {\small TSS-CMC} hypersurfaces family with nonlinear relation between $c$ and $H$,
and both variables range for all real numbers.
This consideration will fulfill the {\small TSS-CMC} foliation with varied mean curvature in each slice in the Kruskal extension.

\fontsize{11}{14pt plus.5pt minus.4pt}\selectfont

\vspace*{5mm}
\noindent{Kuo-Wei Lee}\\
\noindent{Department of Mathematics, National Taiwan University, Taipei, Taiwan}\\
\noindent{E-mail: \verb+d93221007@gmail.com+}
\end{document}